\documentclass[12pt]{amsart}

\usepackage{amssymb}
\usepackage[T1]{fontenc}
\newtheorem{lemma}{Lemma}
\newtheorem{thm}{Theorem}
\newtheorem{prop}{Proposition}
\newtheorem{dfn}{Definition}

\newcommand{\hs}{\hspace{1pt}}
\newcommand{\hsp}{\hspace{-1pt}}

\newcommand{\rt}{\triangleright}

\newcommand{\Lq}{L}
\newcommand{\Mq}{M}

\newcommand{\dom}{\mathrm{Dom}(D)}

\newcommand{\SU}{{\mathrm{SU}}}

\def\C{{\mathbb C}}
\def\N{{\mathbb N}}
\def\R{{\mathbb R}}

\newcommand{\A}{\mathcal{A}}
\newcommand{\B}{\mathcal{B}}
\newcommand{\U}{\mathcal{U}}  \newcommand{\Uq}{\mathcal{U}}
\newcommand{\Hi}{\mathcal{H}}
\renewcommand{\O}{\mathcal{O}}

\newcommand{\M}{\mathcal{M}}
\newcommand{\re}{\mathrm{Re}\,}
\renewcommand{\S}{\mathcal{H}}

\newcommand{\pa}{\k{a}}

\begin{document}
\title[The
fundamental Hochschild class on the
Podle\'s sphere]{A residue formula for the
fundamental Hochschild class on the
Podle\'s sphere}
\author[Ulrich Kr\"ahmer]{Ulrich Kr\"ahmer}
\email{ukraehmer@maths.gla.ac.uk}
\address{University of Glasgow,
Dept.~of Mathematics, University 
Gardens, Glasgow G12 8QW, Scotland}
\author[Elmar Wagner]{Elmar Wagner}
\email{elmar@ifm.umich.mx}
\address{Universidad Michoacana de San Nicol\'as de Hidalgo, 
Instituto de F\'isica y Matem\'aticas,  Edificio C-3, Cd.~Universitaria, 58040 Morelia, Michoacan, M\'exico }
\begin{abstract}
The fundamental Hochschild cohomology
class of the standard Podle\'s quantum
sphere is expressed in terms of the spectral triple
of D\pa browski and Sitarz by
 means of a residue formula.    
\end{abstract}
\maketitle

\section{Introduction}
In the last decade 
many contributions have enhanced the understanding of
how quantum groups and their homogeneous
spaces can be studied in terms of
spectral triples, see e.g.~
\cite{cp,connes1,triesteCP2,trieste,ds,uli,mnw,nt,voigt,sw,W}
and the references therein. But still
some basic questions remain
untouched, e.g.~in how far spectral
triples generate the
fundamental Hochschild cohomology class of
the underlying algebra. This is what we investigate
here for the standard  
Podle\'s quantum sphere \cite{Po}.

To be more precise, the
coordinate ring $\A=\O(\mathrm{S}^2_q)$ is known
to satisfy
Poincar\'e duality in Hochschild
(co)homology, that is, we have
\begin{equation}\label{pd}
		  H^i(\A,\M) \simeq
		  H_{\mathrm{dim}(\A)-i}(\A,\omega
		  \otimes_\A \M),
\end{equation} 
where $\M$ is an $\A$-bimodule and
$\omega=H^{\mathrm{dim}(\A)}(\A,\A
\otimes \A^\mathrm{op})$
\cite{uliisrael}. In the concrete case
of the standard Podle\'s sphere we have $\mathrm{dim} (\A)=2$ and
$\omega \simeq {}_\sigma \A$,
the bimodule obtained from $\A$ by deforming the
canonical left $\A$-action to $a
\triangleright b:=\sigma (a)b$ for some
noninner automorphism $\sigma \in
\mathrm{Aut}(\A)$ [ibid.].
Thus $\A$ is not a
Calabi-Yau algebra (i.e.~$\A \not\simeq \omega$) and
hence the fundamental Hochschild homology class
that corresponds under the isomorphism (\ref{pd})
to $1 \in H^0(\A,\A)$ belongs to
$H_2(\A,{}_\sigma \A)$.

So when switching to the dual picture of
functionals on the canonical
Hochschild 
complex, there is a cocycle
$\varphi : \A^{\otimes 3} \rightarrow
		  \mathbb{C}$,
$$
		  \varphi (a_0a_1,a_2,a_3)
		  -\varphi (a_0,a_1a_2,a_3)
		  +\varphi (a_0,a_1,a_2a_3)
		  -\varphi (\sigma (a_3)a_0,a_1,a_2)=0
$$
whose cohomology class in
$(H_2(\A,{}_\sigma
\A))^* \simeq H^2(\A,({}_\sigma \A)^*)$
is canonically determined by $\A$ (up to
multiplication by invertible elements in
the centre of $\A$ which in our example
consists only of the scalars), and it is
natural to ask whether this class that we refer to as the fundamental Hochschild cohomology class of $\A$ can be
expressed in terms of a spectral triple
over $\A$. Our main result is that this
is indeed the case:
\begin{thm}\label{main}
Let $q \in (0,1)$, $(\A,\Hi,D,\gamma)$ be
the ${\Uq}_q(\mathfrak{su}(2))$-equi\-va\-riant even spectral triple over the standard
Podle\'s quantum sphere con\-struc\-ted by
D\pa browski and Sitarz 
 \cite{ds}, $K$ be
 the standard group-like generator of
 ${\Uq}_q(\mathfrak{su}(2))$, and
 $a_0,a_1,a_2 \in \A$. Then we have:
\begin{enumerate}
\item The operator
$
\gamma a_0[D,a_1][D,a_2]K^{-2}|D|^{-z}
$
is for $\re z>2$ of trace class
and the function 
$$
		  \mathrm{tr}_\Hi(\gamma
		a_0[D,a_1][D,a_2]K^{-2}|D|^{-z})
$$ 
has a  meromorphic continuation to 
$\{z \in \mathbb{C}\,|\,\re z>1\}$
with a pole at
 $z=2$
of order at most 1. 
\\[-8pt]
\item The residue
$$
		  \varphi
 (a_0,a_1,a_2):=\underset{z=2}{\mathrm{Res}}\,
 		  \mathrm{tr}_\Hi(\gamma a_0[D,a_1][D,a_2]K^{-2}|D|^{-z})
$$
defines a Hochschild cocycle that represents 
the fundamental Hochschild cohomology class of $\A$.
\end{enumerate} 
\end{thm}

Let us explain the context of the result. The pioneering papers on the
noncommutative geometry of the Podle\'s
sphere were \cite{mnw}, 
where Masuda, Nakagami and Watanabe
computed $HH_\bullet(\A),HC_\bullet(\A)$
and the K-theory of the
$C^*$-completion of $\A$, and \cite{ds}, where
D\pa browski and Sitarz found the
spectral triple that we use here.  Schm\"udgen and the second author 
then gave a residue formula for a
cyclic cocycle  \cite{sw} that looks like
the one from Theorem~\ref{main}, only
that $K^{-2}$ is replaced by $K^2$. However, Hadfield later computed
the Hochschild and cyclic
homology of $\A$ with coefficients in
${}_\sigma \A$ and deduced that
the cocycle from \cite{sw} is trivial
as a Hochschild cocycle 
\cite{tom}. Finally, the first
author has recently used the cup and cap 
products between the Hochschild
(co)homology groups $H^n(\A,{}_\sigma
\A)$ \cite{uliarab} to produce the surprisingly
simple formula 
\begin{equation}\label{surpris}
\tilde\varphi(a_0,a_1,a_2)=\varepsilon (a_0)E(a_1) F(a_2)
\end{equation} 
for a nontrivial
Hochschild 2-cocycle on $\A$.
Here $\varepsilon$ is the counit of the quantum
$\SU(2)$-group $\B:=\O(\SU_q(2))$ in which
$\A$ is embedded as a subalgebra, and 
$E$ and $F$ are the standard twisted primitive
generators of ${\Uq}_q(\mathfrak{su}(2))$,
considered here
as functionals on $\B$. What we
achieve in Theorem~\ref{main} is to express (a multiple of) this cocycle
in terms of the spectral triple by a residue formula.

The crucial result is Proposition~\ref{reslemma} in Section~\ref{jetzelet}
from which it follows that 
$$
		  \tau_\mu (a) :=
		  \frac{\underset{z=2|\mu|}{\mathrm{Res}}
		  \mathrm{tr}_\Hi (a K^{2 \mu} |D|^{-z})}
		  {\underset{z=2|\mu|}{\mathrm{Res}}\mathrm{tr}_\Hi (K^{2 \mu} |D|^{-z})} 
$$
defines for all $\mu \in \mathbb{R}$ a
twisted trace on $\A$ that we can 
compute explicitly. In the notation
of \cite{uliarab}, these traces are
given by:\\
\begin{center}
\begin{tabular}{|c|c|c|}
\hline
& & \\[-8pt]
$\mu$ & $\sigma$ & $\tau_\mu$\\[-8pt]
& & \\
\hline
& & \\[-8pt]
$< 0$ & any &$\int_{[1]}=\varepsilon$ \\[-8pt]
& & \\
\hline
& & \\[-8pt]
$0$ & $\mathrm{id}$ & $\ \int_{[1]} +
			\frac{\ln q}{2(q^{-1}-q) \ln(q^{-1}-q)} \int_{[x_0]}\ $ \\[-8pt]
& & \\
\hline
& & \\[-8pt]
$\ > 0\ $ & $\ A \mapsto A$, \,$B \mapsto q^{2 \mu}B\ $
	  &$\int_{[1]}-\frac{1-q^{-2
			\mu}}{q(1-q^{-2(\mu + 1)})}\int_{[x_0]}$  \\[-8pt]
& & \\
\hline
\end{tabular}
\end{center}
\mbox{ }\\
Here $\sigma$ is the involved twisting
automorphism and $A$, $B$ are certain
generators of $\A$ (in the notation of \cite{uliarab},
$A=-q^{-1}x_0$). From this fact it will
be deduced in the final section that the
multilinear functional defined in
Theorem~\ref{main} is up to
normalisation indeed cohomologous to the
Hochschild cocycle
(\ref{surpris}) constructed in \cite{uliarab}.

The remainder of the paper is divided
into two sections: Section~\ref{background}
contains background material taken
mainly from \cite{ds,sw}. The subsequent
section discusses the meromorphic
continuation of the zeta functions
$\mathrm{tr}_{\Hi_\pm} (a K^{2 \mu}|D|^{-z})$,
$a \in \A$, and how one can compute their
residues by replacing certain algebra
elements of $\A$ by simpler operators. 
The proof of
Theorem~\ref{main} fills the final
Subsection~\ref{beweis} of the paper.\\

\thanks{The first  author would like to 
thank Viktor Levandovskyy for his
explanations of some of its
functionalities of the computer
algebra system SINGULAR:PLURAL
\cite{plural} which is very capable in carrying out
computations with algebras like the
Podle\'s sphere, see \cite{uliplural} for a quick
demonstration. 
This work was partially supported by 
the EPSRC
fellowship EP/E/043267/1, the
Polish Government Grant N201 1770 33, 
the Marie Curie PIRSES-GA-2008-230836
network, and the Mexican Government Grant 
PROMEP/103.5/09/4106, UMSNH-PTC-259.}

\section{Background}\label{background}
\subsection{The algebras 
\boldmath{$\O(\mathrm{S}^2_q),\ \O(\SU_q(2)),\ {\Uq}_q(\mathfrak{su}(2))$}}\label{stpodl} 
We retain all notations and conventions
used in \cite{sw}. In particular, we fix
a deformation parameter $q \in
(0,1)$, and let $\A=\O(\mathrm{S}^2_q)$ be the
$*$-algebra (over $\mathbb{C}$)
with generators $A=A^*$, $B$ and
$B^*$ and defining relations
$$
		  B A=q^2 AB,\quad
		  A B^*=q^2 B^*A,\quad
		  B^*B=A-A^2,\quad
		  BB^*=q^2A-q^4A^2.
$$

We consider $\A$ as a subalgebra of the
quantised coordinate ring
$\B=\O(\SU_q(2))$ which is the
$*$-algebra generated by $a$, $b$, $c=-q^{-1}b^*$, $d=a^*$
satisfying the relations given e.g.~in
\cite[Eqs.~(1) and (2) on p.~97]{chef}. 
The embedding is given by 
$B=ac$ and $A=-q^{-1}bc$, and it follows that 
$\varepsilon(A)=\varepsilon(B)=\varepsilon(B^*)=0$, 
where $\varepsilon$ denotes the counit of 
$\O(\SU_q(2))$.

Note that it follows from the defining
relations that the monomials 
$$
		  \{A^nB^m,A^n{B^*}^m\,|\,n,m \ge 0\}
$$ 
form a vector space basis of $\A$.

For the Hopf *-algebra
$\U={\Uq}_q(\mathfrak{su}(2))$, we use  
generators $K$, $K^{-1}$, $E$ and $F$ with 
involution $K^*=K$, $E^*=F$, 
defining relations 
$$
KE=qEK,\quad  KF=q^{-1}FK,\quad  EF-FE= \frac{K^2-K^{-2}}{q-q^{-1}},
$$
coproduct 
$$
\Delta(K)=K\otimes K, \ \ \Delta(E)=E\otimes K+K^{-1}\otimes E,
\ \ \Delta(F)=F\otimes K+K^{-1}\otimes F,
$$
and counit $\varepsilon(1-K)=\varepsilon(E)=\varepsilon(F)=0$. 

There is a left $\U$-action on $\A$ satisfying 
$f\rt (ab)= (f_{(1)}\rt a)(f_{(2)}\rt b)$ 
and $f\rt 1=\varepsilon(f)1$ 
for $f\in\U$ and $a,b\in\A$, that is, 
$\A$ is a left $\U$-module algebra. 
Here and in what follows, we use Sweedler's notation 
$\Delta(f)=f_{(1)}\otimes f_{(2)}$. 
On the re-parametrised generators 
\begin{equation}                                                \label{xAB}
 x_{-1}=(1+q^{-2})^{1/2}B, \quad x_0=1-(1+q^2)A, \quad x_{1}=-(1+q^2)^{1/2}B^*,
\end{equation}
this action is given by 
$$
K\rt x_i = q^i x_i,\quad E\rt x_i = (q+q^{-1}) x_{i+1}, \quad 
F\rt x_i = (q+q^{-1}) x_{i-1},
$$
where it is understood that $x_{2}=x_{-2}=0$.

\subsection{The spectral triple}
Our calculations involve the spectral
triple constructed by
D\pa browski and Sitarz in \cite{ds}. For
the reader's convenience and to fix
notation, we recall its definition.

First of all, the $*$-algebra $\A$ becomes represented
by bounded operators on  a Hilbert space
$\Hi:=\S_-\oplus \S_+$ with orthonormal
basis 
$$
		  v^l_{k,\pm} \in \Hi_\pm,\quad 
		  l= \mbox{$\frac{1}{2}$}, \mbox{$\frac{3}{2}$}, \ldots,\quad 
		  k=-l, -l+1,\ldots, l 
$$
where the 
generators $x_{-1}$, $x_0$, $x_1$ act by 
\begin{equation}                                                 \label{pi}
x_i v^l_{k,\pm} =
\alpha^-_i(l,k)_\pm\hspace{1pt} v^{l-1}_{k+i,\pm} + 
\alpha^0_i(l,k)_\pm\hspace{1pt} v^{l}_{k+i,\pm} + 
\alpha^+_i(l,k)_\pm\hspace{1pt} v^{l+1}_{k+i,\pm}.
\end{equation}  
Here 
$\alpha^\nu_i(l,k)_\pm \in \mathbb{R}$
are coefficients 
that can be found e.g.\ in \cite{DDLW}, 
where similar conventions are used. 
We will only need the formulas for
$\alpha_0^\nu(l,k)_\pm$ which are given by 
\begin{align}
		  \alpha^-_0(l,k)_\pm 
&=  \frac{q^{k\pm 1/2} 
		  [2]_q [l\!-\!k]_q^{1/2} [l\!+\!k]_q^{1/2}[l\!-\!1/2]_q^{1/2}
		  [l\!+\!1/2]_q^{1/2}}
		  {[2l\!-\!1]_q^{1/2} [2l]_q
		  [2l\!+\!1]_q^{1/2}},                                      
		  \label{x0-}\\
		  \alpha^0_0(l,k)_\pm 
&= [2l]_q^{-1}
		  \big([l\!-\!k\!+\!1]_q[l\!+\!k]_q - 
		  q^{2} [l\!-\!k]_q[l\!+\!k\!+\!1]_q\big)
		  \beta_\pm (l), 
		  \label{x00} \\
		  \alpha^+_0(l,k)_\pm 
&= \alpha^-_0(l+1,k)_\pm                    
\end{align}
with
\begin{equation}\label{qzahl}
		  [n]_q:=\frac{q^n-q^{-n}}{q-q^{-1}}
\end{equation} 
and
\begin{align}\label{beta+-}                            
		  \beta_\pm(l) 
&= \frac{\pm q^{\mp 1} +  
		  (q\hsp -\hsp
		  q^{-1})([1/2]_q\hspace{1pt}[3/2]_q- 
		  [l]_q[l\!+\!1]_q)
		  }{q[2l+2]_q}.
\end{align}

We now define
$$
		  \dom := \mathrm{span}_\mathbb{C} \{v^l_{k,\pm}\ |\   
		  l= \mbox{$\frac{1}{2}$}, \mbox{$\frac{3}{2}$}, \ldots,\  \,
		  k=-l, -l+1,\ldots, l \}
$$
and on this domain an essentially self-adjoint
operator
$D$ by
$$
		  D v^l_{k,\pm} = [l+1/2]_q v^l_{k,\mp}.
$$

In the sequel all operators we consider
will be defined on this domain, leave
it invariant, and be closable. By slight
abuse of notation we will not
distinguish between an operator defined
on $\dom$ and its closure. 

The $v^l_{k,\pm}$ are eigenvectors
of $|D|$: 
$$
		  |D| v^l_{k,\pm} = [l+1/2]_q v^l_{k,\pm}.
$$
Furthermore,  
the spectral triple is even with grading
$\gamma$ given by
$$
		  \gamma v^l_{k,\pm}=\pm v^l_{k,\pm}.
$$

\subsection{$\U$-equivariance} 
The spectral triple is 
$\U$-equivariant in the sense
of \cite{SitarzBCP}:
On $\dom$ there is 
an action of $\U$ given by 
\begin{equation}\label{K}
		  Kv^l_{k,\pm}=q^k v^l_{k,\pm},
			\quad  Ev^l_{k,\pm}=\alpha^l_k v^l_{k+1,\pm}, 
			\quad Fv^l_{k,\pm}=\alpha^l_{k-1} v^l_{k-1,\pm}, 
\end{equation}
where $\alpha^l_k:=([l-k]_q[l+k+1]_q)^{1/2}$, 
and we have  on $\dom$
$$
		  fa=(f_{(1)}\rt a)f_{(2)},\quad
		  fD=Df,\quad 
		  f\gamma= \gamma f
$$ 
for all $f \in \U$, $a \in \A$. 

Note that the decompostion $\Hi=\Hi_-\oplus \Hi_+$ reduces the representation 
of $\U$ and $\A$ on $\Hi$.

\section{Results} 
\subsection{A family of $q$-zeta functions} 
Quantum group analogues of zeta functions were
studied by several authors, in
particular by Ueno and Nishizawa \cite{ueno},
Cherednik \cite{cherednik}, and Majid
and Toma\v si\'c \cite{majid}. The
ones we will consider here are given on
a suitable domain by
$$
		  \zeta^\pm_T (z) := \mathrm{tr}_{\Hi_\pm}(T|D|^{-z}) 
$$
for some possibly unbounded operator $T$
on $\Hi$. 
The most important case we need is $T=L^\beta K^\delta$ 
for $\beta,\delta \in\mathbb{R}$, where 
$$
		  L v^l_{k,\pm}=q^l v^l_{k,\pm}
$$
and thus 
$$
		  L^\beta K^\delta v^l_{k,\pm}= q^{\beta l+\delta k}
		  v^l_{k,\pm}. 		  
$$

The resulting zeta functions
differ slightly from those considered in
\cite{cherednik,ueno}, and also 
from the one occuring in \cite{sw}. Yet
the main argument leading to a
meromorphic continuation of the functions
to 
the whole complex plane given in
\cite{ueno} can be applied in all these
cases:

\begin{lemma}\label{L137}
For all $\beta,\delta \in \mathbb{R}$, the function 
$$
		  \zeta^\pm_{L^\beta K^\delta}(z)=
		  \sum_{l=\frac{1}{2},\frac{3}{2},\ldots}
			^{\infty} \sum_{k=-l}^{l}  
			\frac{q^{\beta l+\delta k}}
		  {[l+1/2]^z_q},\quad
		  \re z > -\beta+ |\delta|
$$
admits a meromorphic
continuation to the complex plane
given by
\begin{equation*}             
		  \zeta^\pm_{L^\beta K^\delta}(z)=
	  q^{\frac{\beta}{2}}
	  (q^{-\frac{\delta}{2}} \hsp + \hsp q^{\frac{\delta}{2}}) 
		   (1 \hsp - \hsp q^2)^z
		   \sum_{j=0}^\infty		
		  \frac{
		  \binom{z+j-1}{j} \, q^{2j} }
		  {(1 \hsp - \hsp q^{\beta-\delta + 2j + z })
                  (1 \hsp - \hsp q^{\beta+\delta + 2j + z })} 
\end{equation*}
which is holomorphic except at $z=-\beta
 \pm \delta,-\beta \pm \delta -2,\ldots$
and whose residue at
$z=-\beta + |\delta|$ is given by
$$
		  \underset{z=-\beta + |\delta|}
		  {\mathrm{Res}}\;\zeta^\pm_{L^\beta K^\delta} (z)=
		  \left\{\begin{array}{ll}
		  \frac{q^{\frac{\beta-|\delta|}{2}}
		  \left(1-q^2\right)^{|\delta|-\beta}}{\left(q^{|\delta|}-1\right) 
		  \ln (q)}\hs,
\quad & \delta \neq 0,\\[8pt]
		    \frac{2q^{\frac{\beta}{2}}\ln(q^{-1}-q)}
		  {(1-q^2)^\beta (\ln q)^2}\hs,
\quad & \delta = 0.
\end{array}\right.  
$$
\end{lemma}
\begin{proof}
The crucial step is to use the binomial series 
\begin{equation}                    \label{bin}
 		  (1-q^{2(n+1)})^{-z}=
		  \sum_{j=0}^\infty
		 \binom{z+j-1}{j} q^{2(n+1)j}   
\end{equation}
which holds for all $z\in\C$. 

First, let  $\delta = 0$.  
By summing over $k$, 
replacing $l=n+\frac{1}{2}$, 
inserting \eqref{bin}, and interchanging 
the order of the summations in the absolutely
convergent series, 
we obtain for $\re z > -\beta$
\begin{align*}
 \zeta^\pm_{L^\beta} (z)
&=  2(q^{-1}-q)^z \sum_{n = 0}^{\infty} (n+1)
 q^{\beta (n+\frac{1}{2})} q^{(n+1)z} (1-q^{2(n+1)})^{-z}\\
&=
		  2 q^{-\frac{\beta}{2}}(q^{-1}-q)^z 
		\sum_{j=0}^\infty  \sum_{n = 0}^{\infty} (n+1)
		\binom{z\hs +\hs j\hs -\hs 1}{j}
		 q^{(\beta+2j+z) (n+1)}. 
\end{align*}
Using the identity 
$$
		  \sum_{n = 0}^{\infty}
 (n+1)
 t^n=\frac{\mathrm{d}}{\mathrm{d}t} 
 \sum_{n=0}^{\infty} t^n = \frac{1}{(1-t)^2},  
$$
we can write the above sum as 
$$
		  \zeta^\pm_{L^\beta} (z)=
		  2 q^{-\frac{\beta}{2}}(q^{-1}-q)^z 
		  \sum_{j=0}^\infty
		 \left(\begin{array}{c} z+j-1 \\ j
				 \end{array}\right) 
\frac{q^{\beta+2j+z}}{(1-q^{\beta+2j+z})^2}, 
$$
and the right hand side is a
meromorphic function with poles of
second order in
 $z=-\beta,-\beta-2,-\beta-4,\ldots$

If $\delta \neq 0$, the sum over $k$ yields  
$$
		  \sum_{k=-(n+\frac{1}{2})}^{n+\frac{1}{2}} 
		  q^{\delta k} =
		  \frac{q^{-\delta (n+1)}-q^{\delta(n+1)}}
		  {q^{-\frac{\delta}{2}}-q^{\frac{\delta}{2}}}. 
$$  
Similar to the above, we 
get for $\re z > -\beta + |\delta|$
\begin{align*}
&\zeta^\pm_{L^\beta K^\delta} (z) =\\
&
q^{-\frac{\beta}{2}}\frac{(q^{-1}\hsp-\hsp q)^z}{q^{-\frac{\delta}{2}}\hsp-\hsp q^{\frac{\delta}{2}}} 
\sum_{j=0}^\infty  \sum_{n = 0}^{\infty} 
		\!\binom{z \!+ \!j \!-\! 1}{j}
		 (q^{(\beta-\delta+2j+z) (n+1)} \hsp -\hsp  q^{(\beta+\delta+2j+z) (n+1)}). 
\end{align*}
The summation over $n$ gives 
\begin{align*}
 \sum_{n = 0}^{\infty} q^{(\beta-\delta+2j+z) (n+1)}\hsp -\hsp q^{(\beta+\delta+2j+z) (n+1)}=
 \frac{ (q^{-\delta} - q^{\delta})q^{\beta+ 2j + z }}
{(1 \hsp - \hsp q^{\beta-\delta + 2j + z })(1 \hsp - \hsp q^{\beta+\delta + 2j + z })}. 
\end{align*}
Inserting the last equation into the previous one yields the second formula of Lemma \ref{L137} 
which defines a meromorphic function with poles at 
$z=-\beta \hsp \pm\hsp  \delta,-\hsp \beta\hsp\pm\hsp\delta\hsp-\hsp 2,\ldots$

When computing the 
residues at $z=-\beta+|\delta|$, we can
ignore the sum over $j>0$
which is holomorphic in a
neighbourhood of 
$-\beta+|\delta|$. Thus 
$$
		  \underset{z=-\beta+|\delta|}{\mathrm{Res}}\;
		  \zeta^\pm_{L^\beta K^\delta} (z)=
		  \underset{z=-\beta+|\delta|}{\mathrm{Res}}\;
		  \frac{q^{\frac{\beta}{2}}(1-q^2)^z
		  (q^{-\frac{\delta}{2}}+q^{\frac{\delta}{2}})}
		  {(1-q^{\beta-\delta+z})(1-q^{\beta+\delta+z})} 
$$		  
which can be computed straightforwardly
 to yield the result.
\end{proof}

\subsection{A holomorphicity remark}  
Next we need to point out that
$\zeta^\pm_{TL^\beta K^\delta}(z)$ is
holomorphic for $\re z > -\beta+
|\delta|$ whenever $T$ is bounded.
Let us first introduce some notation
that we will use throughout the rest of
the paper in order to simplify
statements and proofs:
\begin{dfn}
We say that a set of complex numbers 
$$\{ \nu_{l,k}\;|\;l\in\mbox{$\frac{1}{2}$}\N,\ \, k=-l,\ldots l\}$$ 
is of order less than or equal to $q^\alpha$, \,$\alpha\in \R$,  
if there exists $C\in(0,\infty)$ such that 
$|\nu_{l,k}|\leq C q^{\alpha l}$ for all
 $k,l$. In this case we write 
$$\nu_{l,k}\precsim q^{\alpha l}.$$ 
\end{dfn}

We refrain from using the
notation $\nu_{k,l}=\mathrm{O}(q^{\alpha l})$ to
avoid confusion about the fact that the
second parameter $k$ can take arbitrary
values from  $\{-l,\ldots,l\}$. Note that
we have for all $\beta, \delta\in \R$ and $z\in\C$
\begin{equation}                                         \label{precsim}
q^{\beta l + \delta k} \precsim q^{(\beta-|\delta|)l},\ \ 
[l-k]_q[l+k]_q\precsim q^{-2l},  \ \ 
[\beta l+\delta]_q^{-z}\precsim q^{
|\beta| \,\re(z)\,l} .
\end{equation}

Now one easily observes:
\begin{lemma}\label{Lq}
For all bounded operators $T$ on $\Hi$ 
 and for all $\beta,\delta \in
 \mathbb{R}$, the function 
$\zeta^\pm_{T{\Lq}^\beta K^\delta}(z)$
is holomorphic on $\{z \in
 \mathbb{C}\,|\,\re z> -\beta+|\delta|\}$.
\end{lemma}   
\begin{proof}
Since $L^\beta K^\delta |D|^{-r}$ is for
$r \in \mathbb{R}$ positive and
 essentially self-adjoint,
the summability of its eigenvalues
verified in
Lemma~\ref{L137} shows that it is of trace
class if $r > - \beta +
|\delta|$. Therefore 
$TL^\beta K^\delta |D|^{-z}=
T|D|^{-\mathrm{i} s} L^\beta K^\delta
 |D|^{-r}$, $z=r+\mathrm{i} s$, is a trace class operator.
 
If one fixes $\epsilon >
 0$, then the infinite series 
defining 
$\mathrm{tr}_{\Hi_\pm}(T {\Lq}^\beta K^\delta |D|^{-z})$
converges uniformly on 
$\{z \in \mathbb{C}\,|\,\re z\geq
 \epsilon-\beta+|\delta|\}$ 
since the geometric series 
$$
		  \sum_{l \in \mathbb{N}}
		  q^{(\re(z)+\beta - |\delta|)l}
$$ 
does so 
and we have for all bounded sequences
 $t^l_{k,\pm}:=\langle v^l_{k,\pm},T
 v^l_{k,\pm} \rangle$
$$
		  \frac{t^l_{k,\pm} q^{\beta l+\delta k}}{[l+1/2]_q^z}
		  \precsim
		  q^{(\re(z)+\beta-|\delta|)l}.
$$  

The partial sums of the series are clearly holomorphic 
functions and, by the above argument,  converge uniformly on compact 
sets contained in  $\{z \in \mathbb{C}\,|\,\re z>
 -\beta + |\delta |\}$. The result follows now from
the Weierstra{\ss} convergence theorem.  
\end{proof}

\subsection{Approximating the generator \boldmath{$A$}}  

It is known \cite{nt} that the
spectral triple we consider here violates Connes'
regularity condition, so the 
standard machinery of zeta functions and
generalised pseudo-differential
operators (see e.g.~\cite{connesmosco,higson})
can not be applied here. 
As a replacement of a pseudo-differential calculus, 
we approximate in the following lemma the 
generator $A \in \A$ on $\Hi$ by simpler operators. 
Similar ideas have been used in \cite{DDLW}. 
\begin{lemma}\label{Mq}
There exists a bounded linear
operator $A_0$ on $\Hi$ 
such that 
\begin{equation}
 A={\Mq}+ A_0{\Lq},\quad \text{where}\ {\Mq} := \Lq^2K^2.
\end{equation}
\end{lemma}   
\begin{proof}
We have to prove that 
$A_0:=(A-{\Mq}){\Lq}^{-1}$ extends to a
bounded operator on $\Hi$. 
Inserting \eqref{xAB} and \eqref{pi}
into this definition shows that  it 
suffices to prove that the coefficients 
$$
		  q^{-l}\alpha^\pm_0(l,k)_\pm\ \ \mbox{and} \ \ 
		  q^{-l}\left(\mbox{$\frac{1}{1+q^2}$}(1-\alpha_0^0(l,k)_\pm)-q^{2(l+k)}\right)
$$
are bounded. 
Applying \eqref{precsim} to \eqref{x0-} gives 
$\alpha^\pm_0(l,k)_\pm\precsim
 q^l$. Therefore we have
$q^{-l}\alpha^\pm_0(l,k)_\pm\precsim 1$ 
which means that these coefficients are bounded. 

Using \eqref{precsim} and
$\frac{1}{1-q^{4l+4}}-1\precsim q^{4l}$, 
we get from \eqref{beta+-} 
$$                                           
		  \beta_\pm(l) =
		  \frac{(q^{-1}\hsp -\hsp q) 
		  [l]_q[l\!+\!1]_q}{q[2l+2]_q}+u_l = 
		  \frac{1-q^{2l}-q^{2l+2}+
		  q^{4l+2}}{1-q^{4l+4}} + u_l=
		  1+v_l,
$$
where $u_l,v_l\precsim
 q^{2l}$. Similarly, we have
\begin{align*}
 \frac{[l\!-\!k\!+\!1]_q[l\!+\!k]_q -
		  q^{2} [l\!-\!k]_q[l\!+\!k\!+\!1]_q}
		  {[2l]_q}
&= \frac{1-(1\!+\!q^2)q^{2l+2k}+(1\!+\!q^2)q^{2l}}{1-q^{4l}}\\[4pt]
&=1 -(1+q^2)q^{2l+2k} + w_{l,k},
\end{align*}
where $w_{l,k}\precsim q^{2l}$. 
Multiplying the last two equations 
and comparing with \eqref{x00} gives
$$
		  \alpha^0_0(l,k)_\pm = 1
		  -(1+q^2)q^{2l+2k} + x_{l,k}
$$ 
with 
$x_{l,k}\precsim q^{2l}$. From this, we get 
$$
		  q^{-l}\left(\mbox{$\frac{1}{1+q^2}$}
		  (1-\alpha_0^0(l,k)_\pm)-q^{2(l+k)}\right)= 
		  q^{-l}x_{l,k}\precsim q^l\precsim 1
$$
which finishes the proof.
\end{proof}

\subsection{Twisted traces as residues}\label{jetzelet}
We are now ready to prove the main technical result
of the paper which expresses certain
twisted traces of $\A$ as residues of
zeta-functions:

\begin{prop}\label{reslemma}
The function $\zeta^\pm_{a K^{2 \mu}}(z)$,
 $a \in \A$, has a
meromorphic continuation to 
$\{z \in
\mathbb{C}\,|\,\re  z> 2|\mu|-1 \}$. 
Its residues at  $z=2|\mu|$ are (both
 for $+$ and $-$) given by\\
\phantom{X} 
\begin{center}
\begin{tabular}{| c | c | c |}
\hline
& & \\[-4pt]
$a$ & $\mu$ & $\underset{z=2|\mu|}
{\mathrm{Res}}\zeta^\pm_{a K^{2 \mu}}(z)$\\[12pt]
\hline
& & \\[-8pt]
$\ A^nB^m,\ A^nB^{*m},\ \ n \ge 0,\ m>0\ $
& {\rm any} 
& $0$\\[-8pt]
& & \\
\hline
& & \\[-8pt]
$A^n,\ \ n>0$ 
& $< 0$ 
& $0$\\
& & \\[-8pt]
\hline
& & \\[-8pt]
$A^n,\ \ n>0$ 
& $\ \geq 0\ $
&$\frac{-q^\mu(q^{-1}-q)^{2\mu}}{(1-q^{2(n+\mu)})\hs\ln q}$\\
& & \\[-8pt]
\hline
& & \\[-8pt]
$1$ 
& $\neq 0$ 
& $\ \frac{-q^{|\mu|}(q^{-1}-q)^{2|\mu|}}{(1-q^{2|\mu|})\,\ln (q)}\ $ \\
& & \\[-8pt]
\hline
& & \\[-8pt]
$1$ & $0$ & $\frac{2\ln(q^{-1}-q)}{(\ln q)^2}$\\[-8pt]
& & \\
\hline
\end{tabular}
\end{center}
\end{prop}
\begin{proof}
Lemma \ref{Lq} implies that the traces  
$\zeta^\pm_{aK^{2\mu}}$ exist
and are holomorphic on  
$\{z \in\mathbb{C}\,|\,\re z> 2|\mu| \}$
for all $a \in \A$.

Furthermore, $B$ and $B^*$ act as shift operators in the 
index $k$ of $v^l_{k,\pm}$. 
Hence the traces $\mathrm{tr}_{\Hi_\pm}(A^n B^mK^{2 \mu}|D|^{-z})$ and 
$\mathrm{tr}_{\Hi_\pm}(A^n B^{*m}K^{2 \mu}|D|^{-z})$ 
va\-nish whenever $m>0$, so we can use 
the trivial
analytic continuation here. Thus it remains to prove 
the claim for $a=A^n$.

Applying first Lemma~\ref{Mq} and using then 
the fact that, by Lemma \ref{Lq}, 
$\mathrm{tr}_{\Hi_\pm}(A^{n-1}A_0{\Mq}^m {\Lq}K^{2\mu}|D|^{-z})$ 
is on $\{z \in\mathbb{C}\,|\,\re z>
 2|\mu|-1 \}$ holomorphic, 
we obtain for all $m \ge 0$, $n>0$    
\begin{align*}
 & \underset{z=2|\mu|}{\mathrm{Res}}\, 
		  \mathrm{tr}_{\Hi_\pm}(A^n {\Mq}^mK^{2\mu}|D|^{-z})\\
&= \!\underset{z=2|\mu|}{\mathrm{Res}}\, 
		  \mathrm{tr}_{\Hi_\pm}(A^{n-1}({\Mq}+A_0{\Lq}) 
		  {\Mq}^mK^{2\mu}|D|^{-z})\\
&=\! \underset{z=2|\mu|}{\mathrm{Res}} 
		  \mathrm{tr}_{\Hi_\pm}\hsp(A^{n-1} {\Mq}^{m+1}K^{2\mu}|D|^{-z})
\hsp +\!\underset{z=2|\mu|}{\mathrm{Res}}
		  \mathrm{tr}_{\Hi_\pm}\hsp(A^{n-1}A_0{\Mq}^m
		  {\Lq}K^{2\mu}|D|^{-z})\\
&=\! \underset{z=2|\mu|}
		  {\mathrm{Res}}
		  \mathrm{tr}_{\Hi_\pm}(A^{n-1} 
		  {\Mq}^{m+1}K^{2\mu}|D|^{-z}). 
\end{align*}
Recall that $M=L^2K^2$. 
An iterated application of the previous equation gives 
\begin{align*}
 		  \underset{z=2|\mu|}
		  {\mathrm{Res}}
		  \mathrm{tr}_{\Hi_\pm}(A^nK^{2\mu}|D|^{-z}) &=
		  \underset{z=2|\mu|}
		  {\mathrm{Res}}
		  \mathrm{tr}_{\Hi_\pm}({\Mq}^nK^{2\mu}|D|^{-z})\\
		  &=\underset{z=2|\mu|}
		  {\mathrm{Res}}
		  \mathrm{tr}_{\Hi_\pm}({\Lq}^{2n}K^{2(\mu+n)}|D|^{-z}).
\end{align*}
The result now reduces to Lemma~\ref{L137}.
\end{proof}

We remark here that the table in the introduction is obtained by comparing 
the values of the twisted traces $\int_{[1]}$ 
and $\int_{[x_0]}$ from \cite{uliarab} 
on the basis vectors $A^nB^m$ 
and  $A^nB^{*m}$ 
with the residues of the last proposition. 

\subsection{Proof of Theorem~\ref{main}}\label{beweis} 
Theorem~\ref{main} is an easy consequence of Proposition \ref{reslemma}. 
As explained e.g.~in \cite{sw}, the operator 
$$
		  \gamma a_0[D,a_1][D,a_2],\quad
		  a_0,a_1,a_2 \in \A
$$
acts by multiplication with
$$
		  a_0(a_1 \triangleleft E)(a_2
		  \triangleleft F) \in \A
$$
on $\Hi_+$ and by multiplication with
$$
		  -a_0(a_1 \triangleleft F)(a_2
		  \triangleleft E) \in \A
$$
on $\Hi_-$. Here $\triangleleft$ denotes
the standard right action of $\U \subset \B^\circ$ on $\B$
given by 
$$
		  a \triangleleft f:=f(a_{(1)})f_{(2)}
$$
Let us point out that unlike the left action $f \triangleright
a:=a_{(1)}f(a_{2})$, this right action does
not leave $\A \subset \B$ invariant, but the
products $(a_1 \triangleleft E)(a_2
\triangleleft F)$ and $(a_1
\triangleleft F)(a_2 \triangleleft E)$
belong to $\A$ again.

By the definition of $\triangleleft$ we have
$$
		  \varepsilon (a_0(a_1 \triangleleft E)(a_2
		  \triangleleft F))=\varepsilon
		  (a_0) E(a_1) F(a_2)
$$
and 
$$
		  \varepsilon (a_0(a_1 \triangleleft F)(a_2
		  \triangleleft E))=
		  \varepsilon (a_0) F(a_1) E(a_2),
$$ 
and evaluation on an arbitrary cycle
representing the fundamental Hochschild
class in $H_2(\A,{}_\sigma \A)$ (see the
proof of the nontriviality of
(\ref{surpris}) in \cite{uliarab}) shows
that the two functionals on
$H_2(\A,{}_\sigma \A)$ induced by these
functionals on
$\A^{\otimes 3}$ coincide up to a factor
of $-q^{-2}$ (see also \cite{uliplural}, where
we carry this computation out with the
help of the 
computer algebra system SINGULAR:PLURAL).

Thus, by Proposition~\ref{reslemma}, 
the cocycle
\begin{align*}
\varphi
 (a_0,a_1,a_2):=&\; 
\underset{z=2}{\mathrm{Res}}\,
 		  \mathrm{tr}_\Hi(\gamma
		  a_0[D,a_1][D,a_2]K^{-2}|D|^{-z})\\
=&\;  \mbox{$\frac{q-q^{-1}}{\ln (q)}$} \hs \varepsilon (a_0) \big(E(a_1)F(a_2) -  F(a_1)E(a_2)\big)
\end{align*}
is cohomologous to  $\frac{q-q^{-3}}{\ln (q)}\hs \tilde \varphi$, where $\tilde \varphi$ denotes 
the fundamental cocycle from \eqref{surpris}. 
This finishes the proof of
Theorem~\ref{main}. \hfill$\Box$


\begin{thebibliography}{99}
\bibitem{cp} Partha
		  Sarathi Chakraborty, Arupkumar
		  Pal, 
		  \emph{On equivariant Dirac
		  operators for ${\SU}_q(2)$}. 
		  Proc. Indian
		  Acad. Sci. Math. Sci.  {\bf 116}  (2006),  no. 4, 531-541. 
\bibitem{cherednik} Ivan Cherednik, 
\emph{On $q$-analogues of Riemann's zeta function}.
Selecta Math. (N.S.) {\bf 7} (2001), no. 4, 447-491.
\bibitem{connes1} Alain Connes, Henri
		  Moscovici, \emph{Type III and
		  spectral triples}.  Traces in number theory, geometry and quantum fields,  57--71, Aspects Math., E38, Friedr. Vieweg, Wiesbaden, 2008.
\bibitem{connesmosco} Alain Connes,
		  Henri Moscovici, \emph{The local
		  index formula in noncommutative
		  geometry}. 
		  Geom. Funct. Anal.  {\bf 5}  (1995),  no. 2, 174-243. 
\bibitem{triesteCP2} Francesco D'Andrea,
		  Ludwik D\pa browski, Giovanni Landi, \emph{The
		  noncommutative geometry of the
		  quantum projective plane}.
		  Rev. Math. Phys.  
		  {\bf 20}  (2008),  no. 8,
		  979-1006.
\bibitem{trieste} Ludwik D\pa browski, \emph{The local index formula for quantum ${\SU}(2)$}.  
Traces in number theory, geometry and quantum fields,  99-110, Aspects Math., E38, Friedr. Vieweg, Wiesbaden, 2008.
\bibitem{DDLW} Ludwik D\pa browski,
		  Francesco D'Andrea, Giovanni
		  Landi, Elmar Wagner,
		  \emph{Dirac operators on all
		  Podle\'s quantum spheres}.
		  J. Noncommut. Geom.  {\bf 1}
		  (2007),  no. 2, 213-239.
\bibitem{ds} Ludwik D\pa browski,
		  Andrzej Sitarz, \emph{Dirac operator on the standard Podle\'s quantum sphere}.  
Noncommutative geometry and quantum groups (Warsaw, 2001),  49-58, 
Banach Center Publ., 61, Polish Acad. Sci., Warsaw, 2003. 
\bibitem{tom} Tom Hadfield,
		  \emph{Twisted cyclic homology of
		  all Podle\'s quantum spheres}.  J. Geom. Phys. {\bf 57}  (2007),  no. 2, 339-351.
\bibitem{higson} Nigel Higson, 
\emph{Meromorphic continuation of zeta
		  functions associated to elliptic
		  operators}. 
Operator algebras, quantization, and noncommutative geometry, 129--142,
Contemp. Math., 365, Amer. Math. Soc., Providence, RI, 2004. 
\bibitem{chef} Anatoli Klimyk, Konrad
		  Schm\"udgen, \emph{Quantum groups
		  and their representations}. Texts
		  and Monographs in
		  Physics. Springer-Verlag,
		  Berlin, 1997. 
\bibitem{uli} Ulrich Kr\"ahmer, 
\emph{Dirac operators on quantum flag manifolds}.
Lett. Math. Phys. {\bf 67} (2004),
		  no. 1, 49-59.
\bibitem{uliisrael} Ulrich Kr\"ahmer, 
\emph{On the Hochschild (co)homology of
		  Quantum Homogeneous Spaces}.  arXiv:0806.0267
\bibitem{uliarab} Ulrich Kr\"ahmer, 
\emph{The Hochschild cohomology ring of
		  the standard Podle\'s quantum
		  sphere}.
		  Arab. J. Sci. Eng. Sect. C Theme
		  Issues {\bf 33}  (2008),  no. 2, 325-335. 
\bibitem{uliplural} http://www.maths.gla.ac.uk/\~{}ukraehmer/podneu.sin
\bibitem{majid} Shahn Majid, Ivan Toma\v
		  si\'c, \emph{On braided zeta
		  functions}, arXiv:1007.5084
\bibitem{mnw} Tetsuya Masuda, Yoshiomi
		  Nakagami, Junsei Watanabe,
		  \emph{Noncommutative
		  differential geometry on the
		  quantum two sphere of Podle\'s
		  I. An algebraic viewpoint}.  
		  $K$-Theory {\bf 5}  (1991),  no. 2, 151-175.
\bibitem{nt} Sergey Neshveyev,  
		  Lars Tuset,
 \emph{A local index formula for
		  the quantum sphere}.
		  Comm. Math. Phys.  {\bf 254}  (2005),
		  no. 2, 323-341.  
\bibitem{voigt} Ryszard Nest, Christian
		  Voigt, \emph{Equivariant
		  Poincar\'e duality for quantum
		  group actions}.
		  J. Funct. Anal. {\bf 258}  (2010),  no. 5, 1466-1503. 
\bibitem{Po} Piotr Podle\'s,
		  \emph{Quantum spheres}.
		  Lett. Math. Phys. {\bf 14}  (1987),  no. 3, 193-202.
\bibitem{sw} Konrad Schm\"udgen, Elmar
		  Wagner, \emph{Dirac operator and
		  a twisted cyclic cocycle on the
		  standard Podle\'s quantum sphere}.
		  J. Reine Angew. Math. {\bf 574}  (2004), 219-235. 
\bibitem{plural} http://www.singular.uni-kl.de
\bibitem{SitarzBCP} Andrzej Sitarz,
		  \emph{Equivariant 
spectral triples}. Noncommutative
		  geometry and quantum groups
		  (Warsaw, 2001),  231-263, Banach
		  Center Publ., 61, Polish Acad. Sci., Warsaw, 2003. 
\bibitem{ueno} Kimio Ueno, Michitomo
		  Nishizawa, \emph{Quantum groups
		  and zeta-functions}.  Quantum groups (Karpacz, 1994),  115-126, PWN, Warsaw, 1995.
\bibitem{W}  Elmar Wagner, \emph{On the noncommutative spin geometry 
             of the standard Podles sphere and index computations}. 
             J. Geom. Phys. {\bf 59} (2009), 998-1016.\\

\end{thebibliography}
\end{document}